\documentclass{amsart}
\usepackage{graphicx}
\usepackage{amsfonts}
\newtheorem{tm}{Theorem}

\newtheorem{defi}{Definition}

\newtheorem{rem}{Remark}

\newtheorem{lm}{Lemma}

\newtheorem{nota}{Notation}

\begin{document}

\title{Polynomials, sign patterns and Descartes' rule of signs}
\author{Vladimir Petrov Kostov}
\address{Universit\'e C\^ote d’Azur, CNRS, LJAD, France} 
\email{vladimir.kostov@unice.fr}
\begin{abstract}
By Descartes' rule of signs, a real degree $d$ polynomial $P$ with all 
nonvanishing coefficients, with $c$ sign changes and $p$ sign preservations 
in the sequence of its coefficients ($c+p=d$) has $pos\leq c$ positive and 
$neg\leq p$ negative roots, where $pos\equiv c($\, mod $2)$ and 
$neg\equiv p($\, mod $2)$. For $1\leq d\leq 3$, 
for every possible choice of the 
sequence of signs of coefficients of $P$ (called sign pattern) 
and for every pair 
$(pos, neg)$ satisfying these conditions there exists a polynomial $P$ 
with exactly 
$pos$ positive and exactly $neg$ negative roots (all of them simple). For 
$d\geq 4$ this is not so. It was observed that for $4\leq d\leq 10$, in all 
nonrealizable cases either $pos=0$ or $neg=0$. It was conjectured that this 
is the case for any $d\geq 4$. We show a counterexample to this conjecture 
for $d=11$. 
Namely, we prove that for the sign pattern $(+,-,-,-,-,-,+,+,+,+,+,-)$ and 
the pair $(1,8)$ there exists no polynomial with $1$ positive, $8$ 
negative simple roots and a complex conjugate pair. 


\end{abstract}
\maketitle 

\section{Introduction}

The classical Descartes' rule of signs says that the real polynomial 
$P(x,a):=x^d+a_{d-1}x^{d-1}+\cdots +a_0$ has not more positive roots 
than the number 
$c$ of sign changes in the sequence of its coefficients. This rule has 
been announced by Ren\'e Descartes (1596-1650)
in his work La G\'eom\'etrie published in 1637. In 1828 Carl Friedrich Gauss 
(1777-1855) has shown that if the roots are counted with multiplicity, then 
the number of 
positive roots has the same parity as $c$. When applied to 
$P(-x)$, these results give an upper bound on the number of negative 
roots of $P$. It is proved in \cite{AlFu} that all possible 
cases (i.e. of $c$, $c-2$, $c-4$, $\ldots$ positive roots) are realizable by 
suitably chosen polynomials $P$ with $c$ sign changes. Notice that here we 
do not impose restrictions on the number of negative roots. 

In what follows we consider polynomials $P$ without zero coefficients. 
Denoting by $p$ the number of sign preservations in the sequence of 
coefficients of $P$, and by $pos_P$ (resp. $neg_P$) the number of 
positive and negative roots of $P$, one can write: 

\begin{equation}\label{posneg}
pos_P\leq c~~~,~~~pos_P\equiv c\, (\, {\rm mod~}2)~~~,~~~neg_P\leq p~~~,~~~
neg_P\equiv p\, (\, {\rm mod~}2)~. 
\end{equation}
We call {\em sign pattern} a finite sequence $\sigma$ of $\pm$-signs; 
we assume that the leading sign of $\sigma$ is $+$. 
For a given 
sign pattern of length $d+1$ with $c$ sign changes and $p$ sign preservations 
we call $(c,p)$ its {\em Descartes pair}, $c+p=d$. 
For a given sign pattern $\sigma$ with Descartes pair $(c,p)$ we call 
$(pos, neg)$ an {\em admissible pair} for $\sigma$ if conditions 
(\ref{posneg}), with $pos_P=pos$ and $neg_P=neg$, are 
satisfied. 

One could ask the question whether given a sign pattern $\sigma$ 
of length $d+1$ 
and an admissible pair $(pos, neg)$ one can find a degree $d$ real monic 
polynomial the signs of whose coefficients define the sign pattern $\sigma$ 
and which has exactly $pos$ simple positive and exactly $neg$ simple negative 
roots. In such a case we say that the couple $(\sigma ,(pos, neg))$ 
is {\em realizable}. 

It 
turns out that for $d=1$, $2$ and $3$ the answer is positive, but for $d=4$ 
the answer is negative; this is due to D.~J.~Grabiner, see \cite{Gr}. 
Namely, for the sign pattern 
$\sigma ^*:=(+,+,-,+,+)$ (with Descartes pair $(2,2)$), 
the pair $(2,0)$ is admissible, 
see (\ref{posneg}), but the couple $(\sigma ^*,(2,0))$ is not realizable. The 
proof of this is easy -- for a monic polynomial $P_5:=x^5+a_4x^4+\cdots +a_0$ 
with signs of the coefficients 
defined by $\sigma ^*$ and having exactly two 
positive roots $u<v$ one has $a_j>0$ for 
$j\neq 2$, $a_2<0$ and $P_5((u+v)/2)<0$. Hence $P_5(-(u+v)/2)<0$ because 
$a_j((u+v)/2)^j=a_j(-(u+v)/2)^j$, $j=0$, $2$, $4$ and 
$0<a_j((u+v)/2)^j=-a_j(-(u+v)/2)^j$, $j=1$,~$3$. As $P(0)=a_0>0$, there are two 
negative roots $\xi <-(u+v)/2<\eta$ as well. 

Modulo the standard $\mathbb{Z}_2\times \mathbb{Z}_2$-action described below, 
Grabiner's example is the only one of nonrealizable couple (sign pattern, 
admissible pair) for $d=4$. The $\mathbb{Z}_2\times \mathbb{Z}_2$-action 
is defined on such couples by two generators. Denote by $\sigma (j)$ the 
$j$th component of the sign pattern $\sigma$. The 
first of the generators replaces the sign pattern $\sigma$ by 
$\sigma ^r$, where $\sigma ^r$ stands for the reverted (i.e. read from 
the back) sign pattern multiplied by $\sigma (0)$, and keeps the same pair 
$(pos, neg)$. This generator corresponds to the fact that the polynomials 
$P(x)$ and $x^dP(1/x)/P(0)$ are both monic and have the same numbers of 
positive and negative roots. The second generator exchanges $pos$ with $neg$ 
and changes the signs of $\sigma$ corresponding to the monomials of odd 
(resp. even) powers if $d$ is even (resp. odd); the rest of the signs 
are preserved. We denote the new sign pattern 
by $\sigma _m$. The generator corresponds to the fact 
that the roots of the polynomials (both monic) $P(x)$ and 
$(-1)^dP(-x)$ are mutually opposite, and if $\sigma$ is the sign pattern of 
$P$, then $\sigma _m$ is the one of $(-1)^dP(-x)$. For a given sign pattern 
$\sigma$ and an admissible pair $(pos, neg)$ the couples 
$(\sigma , (pos, neg))$, $(\sigma ^r, (pos, neg))$, $(\sigma _m, (neg, pos))$ 
and $((\sigma _m)^r, (neg, pos))$ are simultaneously realizable or not. 
(One has $(\sigma _m)^r=(\sigma ^r)_m$.)

All cases of couples (sign pattern, admissible pair) for $d=5$ and $6$ 
which are not realizable are described in \cite{AlFu}. For $d=7$ this is done 
in \cite{FoKoSh} and for $d=8$ in \cite{FoKoSh} and \cite{Ko2}. 
For $d=5$ there is a single nonrealizable case (up to the 
$\mathbb{Z}_2\times \mathbb{Z}_2$-action). The sign pattern is 
$(+,+,-,+,-,-,)$ and the admissible pair is $(3,0)$. For 
$n=6$, $7$ and $8$ there are respectively $4$, $6$, and $19$ 
nonrealizable cases. 
In all of them 
one of the numbers $pos$ or $neg$ is $0$. It is conjectured in 
\cite{FoKoSh} that this is the case for any $d$. The conjecture is based 
on the fact that, using a computer, J. Forsg{\aa}rd has shown that this 
is the case also for $d=9$ and $10$. 

In the present paper we show that the conjecture fails for $d=11$.  

\begin{nota}\label{sigmam}
{\rm For $d=11$ we denote by $\sigma ^0$ the following sign pattern 
(we give on the first and
third lines below respectively the sign
patterns $\sigma ^0$ and $\sigma ^0_m$ while the line in the middle
indicates the positions of the monomials
of odd powers):}

$$\begin{array}{cccccccccccccccc}
\sigma ^0&=&(&+&-&-&-&-&-&+&+&+&+&+&-&)\\ &&&11&&9&&7&&5&&3&&1&&\\ 
\sigma ^0_m&=&(&+&+&-&+&-&+&+&-&+&-&+&+&)\end{array}$$
{\rm In a sense $\sigma ^0$ is centre-antisymmetric -- it consists of 
one plus, five minuses, five pluses and one minus.}
\end{nota}

\begin{tm}\label{maintm}
The sign pattern $\sigma ^0$ is not realizable 
with the admissible pair $(1,8)$.
\end{tm}

The next section contains comments concerning the above result and 
realizability of sign patterns and admissible pairs in general. 
Section~\ref{prelim} contains some technical lemmas which allow to simplify the 
proof of Theorem~\ref{maintm}. The method of proof is explained 
in Section~\ref{method}. Section~\ref{prlm} contains the proofs of lemmas 
used in Section~\ref{method}.

\section{Comments}

Theorem~\ref{maintm} shows that the problem of classifying all nonrealizable 
cases (sign pattern, admissible pair), for any degree $d$, is a difficult one. 
At present, an exhaustive answer in a closed form of a theorem is not known. 
One could 
try to find sufficient conditions for realizability expressed, say, in terms 
of the ratios between $d$, $c$ and $p$. In papers \cite{FoKoSh} and 
\cite{KoSh} series of nonrealizable cases were found 
(defined either for every degree $d$  
or for every odd or even degree sufficiently large). In all of them either 
$pos=0$ or $neg=0$. The construction of such series with $pos\neq 0\neq neg$ 
and the proof of their nonrealizability seems to be sufficiently hard 
given that $d\geq 11$. 

One of the series of nonrealizable cases considered in \cite{FoKoSh} 
concerns sign 
patterns with exactly two sign changes, consisting of $m$ pluses followed by 
$n$ minuses followed by $q$ pluses, $m+n+q=d+1$. Set 

$$\kappa :=\frac{d-m-1}{m}\cdot \frac{d-q-1}{q}~.$$

\begin{lm}\label{lm2changes}
For $\kappa \geq 4$ such a sign pattern is not realizable with the admissible 
pair $(0,d-2)$. The sign pattern is realizable with any admissible pair of 
the form $(2,v)$.
\end{lm}
The lemma is Proposition~6 of \cite{FoKoSh}. 
One of the tools for constructing new realizable cases is the following 
concatenation lemma (proved in \cite{FoKoSh}):

\begin{lm}\label{lmconcat}
Suppose that the
monic polynomials $P_j$ of degrees $d_j$ and with sign 
patterns of the form $(+,\sigma _j)$, $j=1$, $2$ (where $\sigma _j$ contains 
the last $d_j$ components of the corresponding sign pattern) realize the 
pairs $(pos_j, neg_j)$. Then 

(1) if the last position of $\sigma _1$ is $+$, then for any $\varepsilon >0$ 
small enough the polynomial $\varepsilon ^{d_2}P_1(x)P_2(x/\varepsilon )$ 
realizes the sign pattern $(+,\sigma _1,\sigma _2)$ and the pair 
$(pos_1+pos_2,neg_1+neg_2)$;

(2) if the last position of $\sigma _1$ is $-$, then for any $\varepsilon >0$ 
small enough the polynomial $\varepsilon ^{d_2}P_1(x)P_2(x/\varepsilon )$ 
realizes the sign pattern $(+,\sigma _1,-\sigma _2)$ and the pair 
$(pos_1+pos_2,neg_1+neg_2)$ (here $-\sigma _2$ is obtained from $\sigma _2$ 
by changing each $+$ by $-$ and vice versa).
\end{lm} 

It is clear that if Theorem~\ref{maintm} is true, then one should not be able 
to deduce with the help of Lemma~\ref{lmconcat} 
the realizability of the sign pattern 
$\sigma ^0$ with the admissible pair $(1,8)$. Now we show that this is 
indeed impossible. It suffices to check the cases $\deg P_1\geq 6$, 
$\deg P_2\leq 5$ due to the centre-antisymmetry of $\sigma ^0$ and the 
possibility to use the $\mathbb{Z}_2\times \mathbb{Z}_2$-action. 

In all these cases 
the sign pattern of the polynomial $P_1$ has exactly two sign changes (it 
comprises the first sign $+$, the five minuses that follow and the next between 
one and five pluses). These 
cases are (we use the notation from Lemma~\ref{lm2changes}) 
$m=1$, $n=5$, $q=1$, $\ldots$, $5$. The values of $\kappa$ are 
respectively $16$, $10$, $8$, $7$ and $32/5$, all of them are $>4$. 
By Descartes' rule the polynomial $P_1$ 
can have either $0$ or $2$ positive roots. Should it have $2$, then its 
concatenation with $P_2$ should have at least $2$ positive roots 
(by Lemma~\ref{lmconcat}) which is 
impossible. So $P_1$ has no positive roots. The sign patterns defined 
respectively by $P_1$ and $P_2$  
have $4+(q-1)$ and $5-q$ sign preservations. By Lemma~\ref{lm2changes} 
the polynomial $P_1$ has $\leq 2+(q-1)$ negative roots, 
and as $P_2$ has $\leq 5-q$ ones, the concatenation of $P_1$ and $P_2$ has 
$\leq 6$ negative roots. Therefore  
a polynomial (if it exists) realizing the couple $(\sigma ^0,(1,8))$ cannot be 
represented as a concatenation of two polynomials $P_1$ and $P_2$. 

This still does not exclude the existence of such a polynomial. 
In \cite{FoKoSh} certain examples of polynomials realizing given sign 
patterns and admissible pairs had to be constructed directly. Before passing 
to the proof of Theorem~\ref{maintm} we explain the role that the 
concatenation lemma could play in solving the problem of realizability of sign 
patterns with admissible pairs. 

If in the process of solving this problem one arrives to a situation when 
there exists $d_0\in \mathbb{N}$ such that for $d\geq d_0$ 
the realizability of all realizable cases can be deduced from some 
general statements and from the concatenation lemma, 
then it would be sufficient 
to find the exhaustive list of realizable cases for $d<d_0$ and the problem 
would be solved. One could qualify as a general statement 
Lemma~\ref{lm2changes} or the fact that, for even $d$, 
a sign pattern consisting of 
$d+1$ pluses is realizable with the pair $(0,0)$, see~\cite{FoKoSh}, etc. 
The (non)existence of such a degree $d_0$ is not self-evident, 
and if it exists, it is not a priori clear how many new general statements 
of (non)realizability should have to be proved.

\section{Preliminaries\protect\label{prelim}}

\begin{nota} 
{\rm We denote by $S$ the subset of $\mathbb{R}^{11}$ such that if $a\in S$, 
then the signs of the coefficients of the polynomial
$P(x,a)=x^{11}+a_{10}x^{10}+\cdots +a_0$ define the sign pattern $\sigma ^0$ 
and the polynomial $P$ realizes the pair $(1,8)$. 

By $T$ we define the subset of $S$ for which one has $a_{10}=-1$. For a 
polynomial from $S$ one can obtain the conditions $a_{11}=1$, $a_{10}=-1$ by 
rescaling and multiplication by a nonzero constant ($a_{11}$ stands for the 
leading coefficient).}
\end{nota}

\begin{lm}\label{lmnot0}
For $a\in \bar{S}$ one has $a_j\neq 0$ for $j=9,8,7,4,3,2$, and one 
does not have
$a_6=0$ and $a_5=0$ simultaneously.
\end{lm}

Indeed, for $a_j=0$ (where $j$ is one of the
indices $9,8,7,4,3,2$) there are less than $8$ sign changes in the 
sign pattern
$\sigma ^0_m$ hence by Descartes' rule of signs the polynomial $P(.,a)$ 
has less than
$8$ negative roots counted with multiplicity. The same is true for $a_5=a_6=0$. 

\begin{lm}\label{lmnot0bis}
For $a\in \bar{S}$ one has $a_0\neq 0$.
\end{lm}

\begin{rem}
{\rm A priori the set $\bar{S}$ can contain polynomials with all roots real and 
nonzero. The positive ones can be either a triple one or a double 
and a simple ones (but not three simple ones).}
\end{rem}
 
\begin{proof}[Proof of Lemma~\ref{lmnot0bis}:]
Consider first the case $a_j\neq 0$ ($j\neq 0$), $a_0=0$. Hence the polynomial
$P$ has a root at $0$, either $0$ or $2$ positive roots and $8$ negative
roots. Suppose that $P$ has no positive roots. Then the degree $10$ polynomial
$P/x$ defines a sign pattern with exactly two sign changes and has
$8$ negative roots. There exists no such polynomial. Indeed, if it is
with distinct negative roots and with no positive roots, then this would 
contradict Lemma~\ref{lm2changes} (with the notation 
of Lemma~\ref{lm2changes} one has $\kappa =32/5>4$).
If it has $8$ negative roots counted with multiplicity, then one can make them
distinct by a series of perturbations which do not change the signs of the
coefficients of the polynomial, which increase the number of distinct negative 
roots while keeping their total multiplicity equal to $8$ and which 
do not introduce new positive roots as follows. 

Suppose that $P$ has a negative root $-b$ of multiplicity $r$, $1<r\leq 8$. Set 
$P\mapsto P+\varepsilon P_1$,
where $\varepsilon \in (\mathbb{R},0)$, $\varepsilon >0$ and if 
$P=(x+b)^rxQ_1Q_2$, where $Q_1$, $Q_2$ are polynomials, $Q_2$ having
a complex conjugate pair of roots, $Q_1$ having $8-r$ negative roots counted 
with multiplicity,
then $P_1=(x+b)^{r-1}xQ_1$ (this decreases by $1$ the multiplicity of the root 
$-b$ and introduces a new simple negative root).

If the polynomial $P/x$ has two positive roots, then in fact this must be a 
double
positive root $g$ because $a\in \bar{S}$. In this case the perturbations are
with $P_1$ of the form $(x+b)^{r-1}xQ_1(x-g)^2$; after having thus obtained $P$
with $8$ negative simple roots and a double root at $g$ one makes another 
perturbation
$P\mapsto P\pm \varepsilon x$ (the sign before $\varepsilon$ depends on 
whether $P$ has a minimum or maximum at $g$) 
after which the degree $10$ polynomial $P/x$ 
is with $8$ negative
simple roots and no other real root 
which is a contradiction with Lemma~\ref{lm2changes}.

Suppose now that $a_j\neq 0$ ($j\geq 2$) and $a_1=a_0=0$. One considers
in the same way the degree $9$ polynomial $P/x^2$ to obtain a contradiction 
with Lemma~\ref{lm2changes}.
In this case one has $\kappa =7$.

Suppose now that exactly one of the coefficients $a_5$ or $a_6$ is $0$ 
(we assume
that this is $a_5$, for $a_6$ the reasoning is analogous) and
either $a_1\neq 0$, $a_0=0$ or $a_1=a_0=0$ (all other coefficients $a_j$ 
being nonzero).
Then in the perturbations we set $P_1=(x+b)^{r-1}x(x+h_1)(x+h_2)Q_1$, 
where the real
numbers $h_i$ are distinct, different from any of the roots of $P$ 
and chosen such
that the coefficient $\delta$ of $x^5$ of $P_1$ is $0$. 
The choice is possible 
because all coefficients of the polynomial $(x+b)^{r-1}Q_1$ are positive hence 
$\delta$ is of the form $A+(h_1+h_2)B+Ch_1h_2$, where $A>0$, $B>0$ and $C>0$.
\end{proof} 

From now on we consider mainly $T$ (and not $S$) in order not to take into 
account the possibility $a_{10}$ to vanish at some points of $\bar{S}$.
 
\begin{rem}\label{remvanish}
{\rm It results from Lemmas~\ref{lmnot0} and \ref{lmnot0bis} 
that for a polynomial in $\bar{T}$ 
exactly one of the following possibilities exists: 1) all its coefficients 
are nonvanishing; 2) exactly one of them is vanishing, and this is either $a_1$ 
or $a_5$ or $a_6$; 3) exactly two of them are vanishing, and these are either 
$a_1$ and $a_5$ or $a_1$ and $a_6$.}
\end{rem}

\begin{lm}\label{lmnotexist}
There exists no real degree $11$ polynomial the signs of whose 
coefficients define the sign pattern $\sigma ^0$ and which has a single 
positive simple root, negative roots of total multiplicity $8$ and a complex 
conjugate pair with nonpositive real part.
\end{lm}

\begin{proof}
Suppose that such a monic polynomial exists. One can represent it in the form 
$P=P_1P_2P_3$, where $\deg P_1=8$, all roots of $P_1$ are negative hence 
$P_1=\sum _{j=0}^8\alpha _jx^j$, $\alpha _j>0$, $\alpha _8=1$; $P_2=x-w$, $w>0$; 
$P_3=x^2+\beta _1x+\beta _0$, $\beta _j\geq 0$, $\beta _1^2-4\beta _0<0$. 

By Descartes' rule of signs the polynomial $P_1P_2=\sum _{j=0}^9\gamma _jx^j$, 
$\gamma _9=1$,  
has exactly one sign 
change in the sequence of its coefficients. It is clear that 
as $0>a_{10}=\gamma _8+\beta _1$, and as $\beta _1\geq 0$, one must have 
$\gamma _8<0$. But then $\gamma _j<0$ for $j=0$, $\ldots$, $8$. For 
$j=4$, $\ldots$, $8$ one has 
$a_j=\gamma _{j-2}+\beta _1\gamma _{j-1}+\beta _0\gamma _j<0$ which means that 
the signs of $a_j$ do not define the sign pattern $\sigma ^0$.
\end{proof}   

\begin{rem}\label{remnotexist}
{\rm Lemma~\ref{lmnotexist} implies that the set $\bar{T}$ can contain only 
polynomials with negative roots of total multiplicity $8$ and positive roots 
of total multiplicity $1$ or $3$ (i.e. either one simple, 
or one simple and one double 
or one triple positive root), and no root at $0$ (Lemma~\ref{lmnot0bis}). 
Indeed, when approaching the border of $T$, 
the complex conjugate pair can coalesce into a double positive 
(but never nonpositive) root; the latter might eventually coincide with the 
simple positive root.}
\end{rem}
 
\section{The method of proof\protect\label{method}}

Consider $\mathbb{R}^{10}$ as the space 
of the coefficients of the polynomial $P(x,a)|_{a_{10}=-1}$. 
Suppose that there exists a monic polynomial $P(x,a^*)$ 
with signs of its coefficients 
as defined by the sign pattern $\sigma ^0$ (with $a_{10}=-1$), 
with $8$ distinct negative, 
a simple positive and two complex conjugate roots. Then for $a$ close to 
$a^*\in \mathbb{R}^{10}$ all polynomials $P(x,a)$ share with $P(x,a^*)$ 
these properties. Therefore the interior of the set $T$ is nonempty. 
In what follows we denote by $\Gamma$ the connected component of 
$T$ to which $a^*$ belongs. 
Denote by $-\delta$ the value of $a_9$ for $a=a^*$ (recall that this value 
is negative). 

\begin{lm}\label{lmKd} 
There exists a compact set $K\subset \bar{\Gamma}$ containing all points of 
$\bar{\Gamma}$ with $a_9\in [-\delta ,0)$. Hence there exists $\delta _0>0$ 
such that for every point of $\bar{\Gamma}$ one has 
$a_9\leq -\delta _0$, and for at least one point of $K$ and for no point 
of $\bar{\Gamma}\backslash K$ does one have $a_9=-\delta _0$.
\end{lm}

\begin{proof}
Suppose that there exists an unbounded sequence $\{ a^n\}$ of values of 
$a\in \bar{\Gamma}$ with $a_9^n\in [-\delta ,0)$. Hence one can perform 
rescalings $x\mapsto \beta _nx$, $\beta _n>0$, such that the largest of the 
moduli 
of the coefficients of the monic 
polynomials $Q_n:=(\beta _n)^{-11}P(\beta _nx,a^n)$ 
equals $1$. These polynomials belong to $\bar{S}$, not necessarily to 
$\bar{T}$ because $a_{10}$ after the rescalings, in general, 
is not equal to $-1$.
The coefficient of $x^9$ in $Q_n$ equals $a_9^n/(\beta _n)^2$. 
The sequence $\{ a^n\}$ being unbounded there exists a subsequence 
$\beta _{n_k}$ tending to $\infty$. 
This means that the sequence of monic polynomials $Q_{n_k}\in \bar{S}$ with 
bounded coefficients has as one of its limit points a polynomial 
in $\bar{S}$ with $a_9=0$ which contradicts Lemma~\ref{lmnot0}. 

Hence the tuple of coefficients $a_j$ of $P(x,a)\in \bar{\Gamma}$ 
with $a_9\in [-\delta ,0)$ remains 
bounded (hence the same holds true for the moduli of the roots of $P$) from 
which the existence of $K$ and $\delta _0$ follows. 
\end{proof}

The above lemma implies the existence of a polynomial $P_0\in \bar{\Gamma}$ 
with $a_9=-\delta _0$. We say that $P_0$ is $a_9$-{\em maximal}. Our aim is 
to show that no polynomial of $\bar{\Gamma}$ is $a_9$-maximal which 
contradiction will be the proof of Theorem~\ref{maintm}. 

\begin{defi}
{\rm A real univariate polynomial is {\em hyperbolic} if it has only real 
(not necessarily simple) roots. We denote by $H\subset \bar{\Gamma}$ the set 
of hyperbolic polynomials in $\bar{\Gamma}$. Hence these are monic degree 
$11$ polynomials having positive and negative roots of respective total 
multiplicities $3$ and $8$ (zero roots are impossible by Lemma~\ref{lmnot0}). 
By $U\subset \bar{\Gamma}$ we denote the set of 
polynomials in $\bar{\Gamma}$ having a complex conjugate pair, 
a simple positive 
root and negative roots of total multiplicity $8$. 
Thus $\bar{\Gamma}=H\cup U$ and $H\cap U=\emptyset$. We denote by 
$U_0$, $U_2$, $U_{2,2}$, $U_3$ and $U_4$ the subsets of $U$ for which the 
polynomial $P\in U$ has respectively $8$ simple negative roots, one double 
and $6$ simple negative roots, 
at least two negative roots of multiplicity $\geq 2$, one triple and 
$5$ simple negative roots 
and a negative root of multiplicity $\geq 4$.}
\end{defi}

The following lemma on hyperbolic polynomials 
will be used further in the proofs.

\begin{lm}\label{lmhyp}
Suppose that $V$ is a degree $d\geq 2$ 
hyperbolic polynomial with no root at $0$. Then: 

(1) $V$ does not have two or more 
consecutive vanishing coefficients. 

(2) If $V$ has a vanishing coefficient, then the signs of 
its surrounding two coefficients are opposite.

(3) The number of  
positive (of negative) roots of $V$ is equal to the number of sign changes 
in the sequence of its coefficients (in the one of $V(-x)$). 
\end{lm}

The proofs of the lemmas of this section except Lemma~\ref{lmKd} 
are given in 
Sections~\ref{prlm} (Lemmas~\ref{lmhyp} -- \ref{lmH3bis}), \ref{prlmbis} 
(Lemma~\ref{lmH4bis}) and \ref{prlmter} (Lemmas~\ref{lmH4ter} -- \ref{lmH6}).

\begin{lm}\label{lm24}
(1) No polynomial of $U_{2,2}\cup U_4$ is $a_9$-maximal.

(2) For each polynomial of $U_3$ there exists a polynomial of $U_0$ 
with the same values of $a_9$, $a_6$, $a_5$ and $a_1$.

(3) For each polynomial of $U_0\cup U_2$ there exists a polynomial of 
$H\cup U_{2,2}$ with the same values of $a_9$, $a_6$, $a_5$ and $a_1$.
\end{lm} 

The lemma implies that if there exists an $a_9$-maximal polynomial in 
$\bar{\Gamma}$, 
then there exists such a polynomial in $H$. So from now on we aim 
at proving that $H$ contains no such polynomial hence $H$ and $\bar{\Gamma}$ 
are empty.

\begin{lm}\label{lmH2}
There exists no polynomial in $H$ having exactly two distinct real roots.
\end{lm}

\begin{lm}\label{lmHtriple}
The set $H$ contains no polynomial having one triple positive root and  
negative roots of total multiplicity $8$.
\end{lm}

Hence a polynomial in $H$ (if any) has a double and a simple positive roots 
and negative roots of total multiplicity~$8$.
 
\begin{lm}\label{lmH3}
There exists no polynomial $P\in H$ having exactly three distinct real roots 
and satisfying the conditions $\{ a_1=0, a_5=0\}$ or 
$\{ a_1=0, a_6=0\}$.
\end{lm}

It follows from the lemma and from Lemma~\ref{lmnot0} that a polynomial  
$P\in H$ having exactly three distinct roots can satisfy at most one of the 
conditions $a_1=0$, $a_5=0$ and 
$a_6=0$. 

\begin{lm}\label{lmH3bis}
No polynomial in $H$ having exactly three distinct real roots is $a_9$-maximal.
\end{lm}

Thus an $a_9$-maximal polynomial in $H$ (if any) must have at least four 
distinct real roots.

\begin{lm}\label{lmH4bis}
The set $H$ contains no polynomial having a double and a simple positive roots 
and exactly two distinct negative roots of total multiplicity $8$, and which 
satisfies either the conditions $\{ a_1=a_5=0\}$ or $\{ a_1=a_6=0\}$.
\end{lm}

\begin{lm}\label{lmH4ter}
The set $H$ contains no $a_9$-maximal 
polynomial having exactly four distinct real roots 
and satisfying exactly one or none of the conditions 
$a_1=0$, $a_5=0$ or $a_6=0$.
\end{lm}

Therefore an $a_9$-maximal polynomial in $H$ (if any) must have at least five  
distinct real roots.

\begin{lm}\label{lmH5}
The set $H$ contains no $a_9$-maximal 
polynomial having exactly five distinct real roots.
\end{lm}

\begin{lm}\label{lmH6}
The set $H$ contains no $a_9$-maximal 
polynomial having at least six distinct real roots.
\end{lm}

Hence the set $H$ contains no $a_9$-maximal polynomial at all. It results from 
Lemma~\ref{lm24} that there is no such polynomial in $\bar{\Gamma}$. 
Hence~$\bar{\Gamma}=\emptyset$.

\section{Proofs of Lemmas~\protect\ref{lmhyp}, 
\protect\ref{lm24}, 
\protect\ref{lmH2}, \protect\ref{lmHtriple}, 
\protect\ref{lmH3} and \protect\ref{lmH3bis}
\protect\label{prlm}}

\begin{proof}[Proof of Lemma~\ref{lmhyp}:]
Prove part (1). Suppose that a hyperbolic polynomial $V$ 
with two or more vanishing coefficients 
exists. If $V$ is degree $d$ hyperbolic, then $V^{(k)}$ is also hyperbolic 
for $1\leq k<d$. Therefore we can assume that $V$ is of the form 
$x^{\ell}L+c$, where $\deg L=d-\ell$, $\ell \geq 3$, 
$L(0)\neq 0$ and $c=V(0)\neq 0$. 
If $V$ is hyperbolic and 
$V(0)\neq 0$, 
then such is also $W:=x^dV(1/x)=cx^d+x^{d-\ell}L(1/x)$ 
and also $W^{(d-\ell )}$ which 
is of the form $ax^{\ell}+b$, $a\neq 0\neq b$. However given that $\ell \geq 3$ 
this polynomial is not hyperbolic. 

For the proof of part (2) we use exactly the same reasoning, but with 
$\ell =2$. The polynomial $ax^2+b$, $a\neq 0\neq b$, is hyperbolic if and 
only if $ab<0$.

To prove part (3) we consider the sequence of coefficients of  
$V:=\sum _{j=0}^dv_jx^j$, $v_0\neq 0\neq v_d$. 
Set $\Phi :=\sharp \{ k|v_k\neq 0\neq v_{k-1},v_kv_{k-1}<0\}$, 
$\Psi :=\sharp \{ k|v_k\neq 0\neq v_{k-1},v_kv_{k-1}>0\}$ and 
$\Lambda :=\sharp \{ k|v_k=0\}$. Then $\Phi +\Psi +2\Lambda =d$. By Descartes' 
rule of signs the number of 
positive (of negative) roots of $V$ is $pos_V\leq \Phi +\Lambda$ 
(resp. $neg_V\leq \Psi +\Lambda$). As $pos_V+neg_V=d$, one must have 
$pos_V=\Phi +\Lambda$ and $neg_V=\Psi +\Lambda$. There remains to notice that 
$\Phi +\Lambda$ is the number of sign changes in the sequence of coefficients 
of $V$ (and $\Psi +\Lambda$ of $V(-x)$), see part (2) of the lemma.

\end{proof}

\begin{proof}[Proof of Lemma~\ref{lm24}:] Prove part (1). 
A polynomial of $U_{2,2}$ or 
$U_4$ respectively is 
representable in the form:

$$P^{\dagger}:=(x+u)^2(x+v)^2S\Delta ~~~\, \, {\rm and}~~~\, \, 
P^*:=(x+u)^4S\Delta ~,$$
where $\Delta :=(x^2-\xi x+\eta )(x-w)$ and $S:=x^4+Ax^3+Bx^2+Cx+D$. 
All coefficients $u$, $u$, $v$, $w$, $\xi$, $\eta$, $A$, $B$, $C$, $D$  
are positive and $\xi ^2-4\eta <0$ (see Lemma~\ref{lmnotexist}); for 
$A$, $B$, $C$ and $D$ this follows from all roots of $P^{\dagger}/\Delta$ 
and $P^*/\Delta$ being negative. (The roots of $x^4+Ax^3+Bx^2+Cx+D$ 
are not necessarily different from $-u$ and $-v$.) We consider 
the two Jacobian matrices 

$$J_1:=(\partial (a_{10},a_9,a_1,a_5)/\partial (\xi ,\eta ,w,u))~~~{\rm and}~~~
J_2:=(\partial (a_{10},a_9,a_1,a_6)/\partial (\xi ,\eta ,w,u))~.$$
In the case of $P^{\dagger}$ their determinants equal 

$$\begin{array}{ccll}
\det J_1&=&\Pi (CDv+2CDu+C^2uv+2BDv^2+4BDuv&\\ &&
+2BDu^2+2BCuv^2+BCu^2v
+ADv^3+2ADuv^2&\\ &&+3ADu^2v+Cu^2v^3+ACuv^3+2ACu^2v^2)&,\\ 
\det J_2&=&\Pi (BDv+2BDu+Dv^3+2Duv^2+3Du^2v+BCuv+2ADv^2&\\ &&
+4ADuv+2ADu^2+Cuv^3+2u^2v^2C+2ACuv^2+ACu^2v)&,\\ 
\Pi &:=&-2v(w+u)(-\eta -w^2+w\xi )(\xi u+\eta +u^2)&.\end{array}$$
These determinants are nonzero. Indeed, each of the factors is either a sum of 
positive terms or equals 
$-\eta -w^2+w\xi <-\xi ^2/4-w^2+w\xi =-(\xi /2-w)^2\leq 0$. Thus one can choose 
values of $(\xi ,\eta ,w,v)$ close to the initial one ($u$, $A$, $B$, $C$ and 
$D$ remain fixed) to obtain any values of $(a_{10},a_9,a_1,a_5)$ or 
$(a_{10},a_9,a_1,a_6)$ close to the initial one. In particular, with $a_{10}=-1$, 
$a_1=a_5=0$ or $a_{10}=-1$, $a_1=a_6=0$ while $a_9$ can have values larger than 
the initial one. Hence this is not an $a_9$-maximal polynomial. (The values 
of the coefficients $a_j$, $j=0$, $2$, $3$, $4$, $6$ or $5$, $7$ and $8$ 
can change, but their signs remain the same if 
the change of the value of $(\xi ,\eta ,w,v)$ is small enough.) The same 
reasoning is valid for $P^*$ as well in which case one has 
 
$$\begin{array}{ccll}
\det J_1&=&M(3CD+C^2u+8BDu+3BCu^2+6ADu^2+u^4C+3ACu^3)&,\\ 
\det J_2&=&M(3BD+6u^2D+BCu+8ADu+3u^3C+3ACu^2)&,\\ 
M&:=&-4u^2(w+u)(-\eta -w^2+w\xi )(\xi u+\eta +u^2)&.\end{array}$$

To prove part (2) we observe that if the triple root of $P\in U_3$ 
is at $-u<0$, then in case when $P$ is increasing (resp. decreasing) 
in a neighbourhood of $-u$ 
the polynomial $P-\varepsilon x^2(x+u)$ (resp. $P+\varepsilon x^2(x+u)$), 
where $\varepsilon >0$ 
is small enough, has three simple roots close to $-u$; it belongs to 
$\bar{\Gamma}$, its coefficients $a_j$, $2\neq j\neq 3$, are the same as 
the ones of $P$, the signs of $a_2$ and $a_3$ are also the same. 

For the proof of part (3) we observe first that 1) for $x<0$ 
the polynomial $P$ has four maxima and four minima and 
2) for $x>0$ one of the following 
three things holds true: one has $P'>0$, or there is a double 
positive root $\gamma$ of 
$P'$, or $P'$ has two positive roots $\gamma _1<\gamma _2$ 
(they are both either smaller 
or greater than the positive root of $P$). Suppose first that $P\in U_0$. 
Consider the family of polynomials $P-t$, $t\geq 0$. Denote by $t_0$ 
the smallest value of $t$ for which one of the three things happens: 
$P-t$ has a double negative root $v$ (hence a local maximum), 
$P-t$ has a triple 
positive root $\gamma$ or $P-t$ has a double and a simple positive roots 
(the double one is at $\gamma _1$ or $\gamma _2$). In the second and third case 
one has $P-t_0\in H$. In the first case, if $P-t_0$ has another double 
negative root, then $P-t_0\in U_{2,2}$ and we are done. 
If not, then consider the family of 
polynomials 

$$P_s:=P-t_0-s(x^2-v^2)^2(x^2+v^2)^2=P-t_0-s(x^8-2v^4x^4+v^8)~~~,~~~s\geq 0~.$$
The polynomial $-(x^8-2v^4x^4+v^8)$ has double real roots at $\pm v$ 
and a double complex conjugate pair. It has the same signs of the 
coefficients of 
$x^8$, $x^4$ and $1$ as $P-t_0$ and $P$. The rest of the coefficients of 
$P-t_0$ and $P_s$ are the same. As $s$ increases, the value of $P_s$ 
for every $x\neq \pm v$ decreases, so for some $s=s_0>0$ for the first time 
one has either $P_s\in U_{2,2}$ (another local maximum of $P_s$ 
becomes a double negative root) or $P_s\in H$ ($P_s$ has positive roots 
of total multiplicity $3$, but not three simple ones). This proves part (3) for 
$P\in U_0$.

If $P\in U_2$ and the double negative root is a local minimum, then 
the proof of part (3) is just the same. If this is a local maximum, 
then one skips the construction of the family $P-t$ and starts constructing 
directly the family $P_s$. 
\end{proof}

\begin{proof}[Proof of Lemma~\ref{lmH2}:]
Suppose that such a polynomial exists. Then it must be of the form 
$P:=(x+u)^8(x-w)^3$, $u>0$, $w>0$. The conditions $a_{10}=-1$ and  
$a_1>0$ read:

$$8u-3w=-1~~~\, \, {\rm and}~~~\, \, u^7w^2(3u-8w)>0~.$$
In the plane of the variables $(u,w)$ the domain $\{ u>0, w>0, 3u-8w>0\}$ 
does not intersect the line $8u-3w=-1$ which proves the lemma.
\end{proof}

\begin{proof}[Proof of Lemma~\ref{lmHtriple}:]

Represent the polynomial in the form $P=(x+u_1)\cdots (x+u_8)(x-\xi )^3$, where 
$u_j>0$ and $\xi >0$. The numbers $u_j$ are not necessarily distinct. 
The coefficient $a_{10}$ then equals $u_1+\cdots +u_8-3\xi$. The condition 
$a_{10}=-1$ implies $\xi =\xi _*:=(u_1+\cdots +u_8+1)/3$. Denote by 
$\tilde{a}_1$ the coefficient $a_1$ 
expressed as a function of $(u_1,\ldots ,u_8,\xi )$. Using 
computer algebra (say, MAPLE) one finds $27\tilde{a}_1|_{\xi =\xi _*}$:

$$27\tilde{a}_1|_{\xi =\xi _*}=-(-u_1\cdots u_8+X+Y)(u_1+\cdots +u_8+1)^2~,$$
where $Y:=u_1\cdots u_8(1/u_1+\cdots +1/u_8)$ and 
$X:=u_1\cdots u_8\sum _{1\leq i,j\leq 8,i\neq j}u_i/u_j$ (the sum $X$ contains $56$ 
terms). We show that $a_1<0$ which contradiction proves the lemma. The factor 
$(u_1+\cdots +u_8+1)^2$ is positive. The factor $\Xi :=-u_1\cdots u_8+X+Y$ 
contains a single monomial with a negative coefficient, namely, 
$-u_1\cdots u_8$. Consider the sum 

$$\begin{array}{clcc}
&-u_1\cdots u_8+u_1^2u_3u_4u_5u_6u_7u_8+u_2^2u_3u_4u_5u_6u_7u_8&&\\ 
=&u_3u_4u_5u_6u_7u_8((u_1-u_2)^2+u_1u_2)&>&0\end{array}$$
(the second and third monomials are in $X$). Hence $\Xi$ is representable 
as a sum of positive quantities, so $\Xi >0$ and $a_1<0$. 
\end{proof}

\begin{proof}[Proof of Lemma~\ref{lmH3}:]
Suppose that such a polynomial exists. Then it must be of the  
form $(x+u)^8(x-w)^2(x-\xi )$, 
where $u>0$, $w>0$, $\xi >0$, $w\neq \xi$. 
One checks numerically (say, using MAPLE), 
for each of the two systems of algebraic 
equations 
$a_{10}=-1$, $a_1=0$, $a_5=0$ and $a_{10}=-1$, $a_1=0$, $a_6=0$, that each 
real solution $(u, w, \xi )$ or $(u,v,w)$ contains a nonpositive component. 
\end{proof}

\begin{proof}[Proof of Lemma~\ref{lmH3bis}:]
Making use of Lemma~\ref{lmHtriple} we consider only polynomials of the form 
$(x+u)^8(x-w)^2(x-\xi )$. 
Consider the Jacobian matrix 
$J_1^*:=(\partial (a_{10},a_9,a_1)/\partial (u,w,\xi ))$. Its determinant equals 
$6u^6(u+w)(u-7w)(\xi -w)(k+u)$. All factors except $u-7w$ are nonzero. 
Thus for $u\neq 7w$ one has $\det J_1\neq 0$, so one 
can fix the values of $a_{10}$ and $a_1$ 
and vary the one of $a_9$ arbitrarily close to the initial one by choosing 
suitable values of $u$, $w$ and $\xi$. Hence the polynomial is not 
$a_9$-maximal. For $u=7w$ one has $a_3=-117649w^7(35w+8\xi )<0$ which is 
impossible. Hence 
there exist no $a_9$-maximal polynomials 
which satisfy only the condition $a_1=0$ or none of the conditions $a_1=0$, 
$a_5=0$ or $a_6=0$. To see that there exist no such polynomials satisfying only 
the condition $a_5=0$ or $a_6=0$ one can consider the matrices 
$J_5^*:=(\partial (a_{10},a_9,a_5)/\partial (u,w,\xi ))$ and 
$J_6^*:=(\partial (a_{10},a_9,a_6)/\partial (u,w,\xi ))$. Their determinants 
equal respectively 

$$112u^2(u+w)(5u-3w)(\xi -w)(\xi +u)~~~\, {\rm and}~~~\, 
112u(u+w)(3u-w)(\xi -w)(\xi +u)$$
and are nonzero respectively for $5u\neq 3w$ and $3u\neq w$, in which cases 
in the same way we conclude that the polynomial is not $w_9$-maximal. 
If $u=3w/5$, then $a_1=-(2187/390625)w^9(-3w+34\xi )$ and $a_{10}=-\xi +14w/5$. 
As $a_1>0$ and $a_{10}<0$, one has $w>34\xi /3$ and $\xi >14w/5>(34/3)(14/5)\xi$ 
which is a contradiction. If $w=3u$, then $a_6=14u^4(10u+\xi )>0$ which is 
again a contradiction.

%
\end{proof}

\section{Proof of Lemma~\protect\ref{lmH4bis}
\protect\label{prlmbis}}

The multiplicities of the negative roots of $P$ define the following a priori 
possible cases: 

$${\rm A)}~~~(7,1)~,~~~{\rm B)}~~~(6,2)~,~~~{\rm C)}~~~(5,3)~~~
{\rm and~D)}~~~(4,4)~.$$ 
In all 
of them the proof is carried out simultaneously for the two possibilities 
$\{ a_1=a_5=0\}$ and $\{ a_1=a_6=0\}$. In order to simplify the proof we fix 
one of the roots to be equal to $-1$ (this can be achieved by a change 
$x\mapsto \beta x$, $\beta >0$, followed by $P\mapsto \beta ^{-11}P$). 
This allows to deal with one 
parameter less. By doing so we can no longer require that $a_{10}=-1$, 
but only that $a_{10}<0$.

\begin{proof}[Proof in Case A):]
We use the following parametrisation: $P=(x+1)^7(sx+1)(tx-1)^2(wx-1)$, $s>0$, 
$t>0$, $w>0$, $t\neq w$, i.e. 
the negative roots of $P$ are at $-1$ and $-1/s$ and the positive ones at 
$1/t$ and $1/w$. 

The condition $a_1=w+2t-s-7=0$ yields $s=w+2t-7$. For $s=w+2t-7$ one has 

$$\begin{array}{llllllll}
a_3&=&a_{32}w^2+a_{31}w+a_{30}&,&a_4&=&a_{42}w^2+a_{41}w+a_{40}&,\\ \\ 
&&{\rm where}&&a_{32}&=&-2t+7&,\\ 
a_{31}&=&-(2t-7)^2&,&a_{30}&=&-2t^3+28t^2-98t+112&\\ \\ 
&&{\rm and}&&a_{42}&=&t^2-14t+21&,\\ 
a_{41}&=&2t^3-35t^2+140t-147&,&a_{40}&=&
-14t^3+112t^2-294t+210&.\end{array}$$
The coefficient $a_{30}$ has a single real root $9.436\ldots$ hence $a_{30}<0$ 
for $t>9.436\ldots$. On the other hand 

$$a_{32}w^2+a_{31}w=w(-2t+7)(w+2t-7)=w(-2t+7)s$$ 
which is $<0$ for $t>9.436\ldots$. Thus the inequality $a_3>0$ fails for 
$t>9.436\ldots$. Observing that $a_{41}=(2t-7)a_{42}$ one can write 

$$a_4=(w+2t-7)wa_{42}+a_{40}=swa_{42}+a_{40}~.$$
The real roots of $a_{42}$ (resp. $a_{40}$) equal 
$1.708\ldots$ and $12.291\ldots$ (resp. $1.136\ldots$). Hence 
for $t\in [1.708\ldots ,12.291\ldots ]$ the inequality $a_4>0$ 
fails. There remains to consider the possibility 
$t\in (0,1.708\ldots )$.

It is to be checked directly that for $s=w+2t-7$ one has 

$$a_{10}/t=(7t-2)w(w+2t-7)+t(7-2t)=(7t-2)ws+t(7-2t)$$
which is $\geq 0$ (hence $a_{10}<0$ fails) for $t\in [2/7,7/2]$. 
Similarly 

$$\begin{array}{ll}
a_6=a_6^*w(w+2t-7)+a_6^{\dagger}=a_6^*ws+a_6^{\dagger}~~,~~&{\rm where}\\ 
a_6^*=21t^2-70t+35~~,~~&a_6^{\dagger}=
-70t^3+350t^2-490t+140~.\end{array}$$
The real roots of $a_6^*$ (resp. $a_6^{\dagger}$) equal 
$0.612\ldots >2/7=0.285\ldots$ and $2.720\ldots$ (resp. 
$0.381\ldots >2/7$, $2$ and 
$2.618\ldots$) hence for $t\in (0,2/7)$ one has 
$a_6^*>0$ and $a_6^{\dagger}>0$, i.e. $a_6>0$ and the equality $a_6=0$  
or the inequality $a_6<0$ is impossible.
\end{proof}

\begin{proof}[Proof in Case B):]
We parametrise $P$ as follows: $P=(x+1)^6(Tx^2+Sx-1)^2(wx-1)$, $T>0$, $w>0$.  
In this case we presume $S$ to be real, not necessarily positive. The factor 
$(Tx^2+Sx-1)^2$ contains the double positive and negative roots of $P$. 

From $a_1=w+2S-6=0$ one finds $S=(6-w)/2$. For $S=(6-w)/2$ one has 

$$\begin{array}{rclcl}
a_{10}/T&=&(6w-1)T+6w-w^2&,&\\ a_7&=&a_{72}T^2+a_{71}T+a_{70}&,&{\rm where}\\ 
a_{72}&=&15w-20&,&\\ a_{71}&=&-20w^2+105w-78&{\rm and}&\\ 4a_{70}&=&
15w^3-162w^2+468w-192&.&\end{array}$$
Suppose first that $w>1/6$. The inequality $a_{10}<0$ is equivalent to 
$T<(w^2-6w)/(6w-1)$. As $T>0$, this implies $w>6$. 

For $T=(w^2-6w)/(6w-1)$ one obtains $a_7=3C/4(6w-1)^2$, where the numerator 
$C:=40w^5-444w^4+1345w^3-502w^2+300w-64$ has a single real root $0.253\ldots$. 
Hence for $t>6$ one has $C>0$ and $a_7|_{T=(w^2-6w)/(6w-1)}>0$. On the other hand 
$a_{70}=a_7|_{T=0}$ has roots $0.489\ldots$, $4.504\ldots$ and $5.805\ldots$, so 
for $w>6$ one has $a_7|_{T=0}>0$. For $w>6$ fixed, and for 
$T\in [0,(w^2-6w)/(6w-1)]$, the value of the derivative 

$$\partial a_7/\partial T=(30w-40)T-20w^2+105w-78$$
is maximal for $T=(w^2-6w)/(6w-1)$; this value equals 

$$-(90w^3-430w^2+333w-78)/(6w-1)$$
which is $<0$ because the only real root of the numerator is $3.882\ldots$. 
Thus $\partial a_7/\partial T<0$ and $a_7$ is minimal for $T=(w^2-6w)/(6w-1)$. 
Hence the inequality $a_7<0$ fails for $w>1/6$. For $w=1/6$ one has 
$a_{10}=35T/36>0$. 

So suppose that $w\in (0,1/6)$. In this case the condition $a_{10}<0$ implies 
$T>(w^2-6w)/(6w-1)$. For $T=(w^2-6w)/(6w-1)$ one gets 

$$a_4=3D/4(6w-1)^2~~~,~~~{\rm where}~~~D:=64w^5-300w^4+502w^3-1345w^2+444w-40$$
has a single real root $3.939\ldots$. Hence for $w\in (0,1/6)$ one has 
$D<0$ and $a_4|_{T=(w^2-6w)/(6w-1)}<0$. The derivative 
$\partial a_4/\partial T=-w^2-2T-6$ being negative one has $a_4<0$ for 
$w\in (0,1/6)$, i.e. the inequality $a_4>0$ fails.
\end{proof}

\begin{proof}[Proof in Case C):]
We use the following parametrisation: $P=(x+1)^5(xs+1)^3(xt-1)^2(xw-1)$. 

From $a_1=w+2t-5-3s=0$ one gets $s=(w+2t-5)/3$. For $s=(w+2t-5)/3$ one has 
$27a_{10}=tS(w+2t-5)^2$, where 
\begin{equation}\label{S}
S:=10wt^2-2t^2+5w^2t-21wt+5t-2w^2+10w~.
\end{equation}
The factor $S$ can be represented as a polynomial in $w$ or in $t$; for each of 
the cases we give its discriminant (and the latter's real roots) as well:

$$\begin{array}{lcl}
S=(5t-2)w^2+(10-21t+10t^2)w+5t-2t^2&,&\\ D_1=5(t-2)(2t-1)(10t^2-13t+10)&,& 
0.5~~~,~~~2\\ \\ 
S=(10w-2)t^2+(5w^2-21w+5)t-2w^2+10w&,&\\ D_2=5(w^2-5w+1)(5w^2-w+5)&,& 
0.208\ldots ~~~,~~~4.791\ldots \end{array}$$
Hence for $t\in [0.5,2]$ or for $w\in [0.208\ldots ,4.791\ldots ]$ one has 
respectively $D_1\leq 0$ and $D_2\leq 0$ hence $S\geq 0$ and the inequality 
$a_{10}<0$ fails. The partial derivative

$$\partial S/\partial t=5w^2-21w+20wt-4t+5=5w(w-4.2)+(20w-4)t+5$$
is positive for $t>2$ and $w>4.791\ldots$. Hence $S>0$ for $t>2$ and 
$w>4.791\ldots$. For $(t,w)\in (0,0.5)\times (0,0.208\ldots )$ one has 
$w+2t-5<0$, i.e. $s<0$. Thus Case C) is impossible outside the 
two semi-strips 

$$\Sigma _1:=\{ (t,w)\in (0,0.5)\times (4.791\ldots ,\infty )\} ~~~{\rm and}~~~
\Sigma _2:=\{ (t,w)\in (2,\infty )\times (0,0.208\ldots )\} ~.$$

\begin{lm}
The inequality $a_4>0$ fails on $\Sigma _2$.
\end{lm} 

\begin{proof}Indeed, 

$$\begin{array}{lcl}
27a_4=w^4+s_3w^3+s_2w^2+s_1w+s_0&,&{\rm where}\\ 
s_3=-10t+25&,&s_2=-30t^2+60t-120\\ 
s_1=-22t^3+75t^2-120t+175&&{\rm and}
\\ s_0=-20t^4+110t^3-300t^2+350t-410&&.\end{array}$$
For $(t,w)\in \Sigma _2$ one has 

$$w^4+s_3w^3\leq 0.208\ldots ^4+(-10\times 2+25)\times 0.208\ldots ^3<0.05~.$$
The trinomial $s_2$ is negative (because its discriminant is such), so 
$s_2w^2<0$. The quantity $s_0$ is decreasing for $t\geq 2$ (because the only 
real root of its derivative equals $1$), so in $\Sigma _2$ one has 
$s_0<s_0|_{t=2}=-350$. Finally, the quantity $s_1$ is decreasing 
(its derivative has no real roots) hence in $\Sigma _2$ the term 
$s_1w$ is less than $s_1|_{t=2}w\leq 59\times 0.208\ldots <13$. Thus 
$a_4<0.05-350+13<0$ in $\Sigma _2$.
\end{proof} 

We define the sets 

$$\begin{array}{llllllrl}
\Sigma _3&:=&\{ (t,w)&\in &[0,0.5]&\times &[6.75\ldots ,\infty )\} &,\\ 
\Sigma _4&:=&\{ (t,w)&\in &[0.25,0.5]&\times &[4.791\ldots ,6.75]\} &,\\ 
\Sigma _5&:=&\{ (t,w)&\in &[0,0.25]&\times &[5,6.75]\}&{\rm and}\\ 
\Sigma _6&:=&\{ (t,w)&\in &[0,0.25]&\times &[4.791\ldots ,5]\}&.\end{array}$$
One can observe that 
$\Sigma _1\subset (\Sigma _3\cup \Sigma _4\cup \Sigma _5\cup \Sigma _6)$. 
For $w=6.75$ one has 

$$27a_6=14t^5+511.75t^4-44.09375t^3-6341.949214t^2-4336.44531t+3760.50781$$
Its real roots are $-36.303\ldots$, $-3.058\ldots$, $-1.324\ldots$, 
$0.503\ldots$ and $3.629\ldots$. Hence for $t\in (0,0.5)$, $w=6.75$ one has 
$a_6>0$. One can represent $27\partial a_6/\partial w$ in the form 
$(4w-5+2t)g$, where

$$g:=4t^4+4t^3w+t^2w^2-35t^2-20wt^2+90t-10w^2t+20wt-5-40w+10w^2~.$$
Hence $g|_{w=6.75}=4t^4+27t^3-124.4375t^2-230.625t+180.625$, with real roots 
$-9.360\ldots$, $-1.982\ldots$, $0.610\ldots$ and $3.982\ldots$, so  
$g|_{w=6.75}$ is $>0$ for $t\in (0,0.5)$. 

\begin{lm}The derivative  
$\partial g/\partial w=(2t^2-20t+20)w+4t^3-20t^2+20t-40$
is positive on $\Sigma _3$.
\end{lm} 
Hence this is the case of 
$\partial a_6/\partial w$ and $a_6$ as well, so the inequality $a_6<0$ or 
the equality $a_6=0$ fails of $\Sigma _3$. 

\begin{proof}
On $\Sigma _3$ one has 

$$\begin{array}{cccccccc}
(2t^2-20t+20)w&>&(-20t+20)w&>&10\times 6.75&=&67.5&{\rm and}\\ 
4t^3-20t^2+20t-40&>&4t^3-40&>&-40&,&&\end{array}$$
so $\partial a_6/\partial w>0$.
\end{proof} 

\begin{lm}\label{lmK}
One has $a_{10}\geq 0$ on $\Sigma _4$.
\end{lm}

\begin{proof}
One has $a_{10}=(t/27)(w+2t-5)^2S$, see (\ref{S}), 
hence  
$S|_{t=0.25}=-0.75w^2+5.375w+1.125$ which is positive for 
$w\in [4.791\ldots ,6.75]$. The lemma follows from   
$\partial S/\partial t=(20w-4)t+5w^2-21w+5$ being positive 
for $(t,w)\in \Sigma _4$. 
\end{proof}

\begin{lm}
One has $a_6>0$ in $\Sigma _5$.
\end{lm}

\begin{proof}
We use the following expression for $27a_6$:

$$\begin{array}{llll}
27a_6=h_4w^4+h_3w^3+h_2w^2+h_1w+h_0&,&h_4=t^2-10t+10&,\\ 
h_3=6t^3-35t^2+50t-70&,&h_2=12t^4-30t^3+90t+90&,\\ 
h_1=8t^5-20t^4-70t^3+355t^2-460t+25&,&&\\ 
h_0=-40t^5+100t^4-50t^3-50t^2+50t+260&.&&\end{array}$$
Hence the values for $w=5$ of the derivatives $27\partial ^sa_6/\partial w^s$ 
are the following polynomials:

$$\begin{array}{lll}
27\partial ^0a_6/\partial w^0&=&300t^4-400t^3-2025t^2+135\\ 
27\partial ^1a_6/\partial w^1&=&8t^5+100t^4+80t^3-1770t^2-810t+675\\ 
27\partial ^2a_6/\partial w^2&=&24t^4+120t^3-750t^2-1320t+1080\\ 
27\partial ^3a_6/\partial w^3&=&36t^3-90t^2-900t+780\\ 
27\partial ^4a_6/\partial w^4&=&24t^2-240t+240\end{array}$$
All of them are positive for $t\in [0,0.25]$ from which 
and from the Taylor series of $a_6$ w.r.t. the variable $w$ the lemma follows.
\end{proof} 

\begin{lm}
One has $a_{10}\geq 0$ on $\Sigma _6$.
\end{lm}

\begin{proof}
Recall that the quantity $S$ was defined by (\ref{S}).  
The values for $t=0$ of the derivatives 
$\partial ^sS/\partial t^s$ are:

$$\begin{array}{lll}  
\partial ^0S/\partial t^0&=&-2w^2+10w\\ 
\partial ^1S/\partial t^1&=&5w^2-21w+5\\ 
\partial ^2S/\partial t^2&=&20w-4~.\end{array}$$
They are all nonnegative for $w\in [4.791,5]$ from which 
and from the Taylor series of $S$ w.r.t. the variable $t$ one gets $S\geq 0$ in 
$\Sigma _6$ and the lemma follows.
\end{proof} 
\end{proof}

\begin{proof}[Proof in Case D):]
$P=(x+1)^4(sx+1)^4(tx-1)^2(wx-1)$.\\ 

The condition $a_1=w+2t-4s-4=0$ implies $s=(w+2t-4)/4$. For $s=(w+2t-4)/4$ one 
has 
$256a_{10}=t(w+2t-4)^3H^*$, where 
\begin{equation}\label{H}
H^*:=8wt^2-2t^2+4w^2t-5wt+4t+8w-2w^2~.
\end{equation}
\begin{lm}\label{lmH}
The inequality $H^*\geq 0$ (hence $a_{10}\geq 0$) holds in each of the two cases 
$t\in [1/2,2]$ and $w\in [1/4,4]$. It holds also for 
$(t,w)\in [2,\infty )\times [4,\infty )$, for 
$(t,w)\in (0,1/2]\times (0,1/4]$ and for $(t,w)\in [0.3,1/2]\times [4,6.71]$.
\end{lm}

\begin{rem}
{\rm In other words, for $t>0$, $w>0$, the inequality $a_{10}<0$ fails 
outside the domain $\Omega _1\cup \Omega _2\cup \Omega _3$, where} 

$$\Omega _1:=(2,\infty )\times (0,1/4)~~~,~~~
\Omega _2:=(0,1/2)\times (6.71,\infty )~~~,~~~
\Omega _3:=(0,0.3)\times (4,6.71]~.$$
{\rm We set $\Omega _3=\Omega _3^-\cup \Omega _3^+$, where} 

$$\Omega _3^-:=(0,0.3)\times (4,5]~~~,~~~
\Omega _3 ^+:=(0,0.3)\times (5,6.71]~.$$
\end{rem}

\begin{proof}[Proof of Lemma~\ref{lmH}]
We represent $H^*$ in two ways: 

$$\begin{array}{lclc}
H^*=H_{2w}w^2+H_{1w}w+H_{0w}&,&H_{2w}=4t-2&,\\ 
H_{1w}=8t^2-5t+8&,&H_{0w}=-2t^2+4t&{\rm and}\\ \\ 
H^*=H_{2t}t^2+H_{1t}t+H_{0t}&,&H_{2t}=8w-2&,\\ 
H_{1t}=4w^2-5w+4&,&H_{0t}=-2w^2+8w&.\end{array}$$
The first statement of the lemma follows from 
$H_{jw}\geq 0$, $j=1$, $2$, $3$ for $t\in [1/2,2]$ and 
$H_{jt}\geq 0$, $j=1$, $2$, $3$ for $w\in [1/4,4]$. 
The quantity $H^*$ is a degree 
$2$ polynomial in $t$. For $t=2$ and $w\in [4,\infty)$ one has 

$$\begin{array}{lcll}
H^*=30w+6w^2>0&,&\partial ^2H^*/\partial t^2=16w-4\geq 0&{\rm and}\\ 
\partial H^*/\partial t=16wt-4t+4w^2-5w+4&=&(16w-4)t+w(4w-5)+4>0&,\end{array}$$
so by representing $H^*$ as a Taylor series in the variable $t$ 
we see again that 
$H^*>0$ for $(t,w)\in [2,\infty )\times [4,\infty)$. 
Next, for $(t,w)\in (0,1/2]\times (0,1/4]$ one can write 

$$H^*=t(4-2t-5w)+2w(4-w)+8wt^2+4w^2t>0~.$$
Finally, as $\partial H^*/\partial t=(16w-4)t+4w^2-5w+4$, where the polynomial 
$4w^2-5w+4$ has no real roots, one has $\partial H^*/\partial t>0$ in 
$[0.3,1/2]\times [4,6.71]$. On the other hand for $t=0.3$ the 
polynomial $H^*$ equals $w(7.22-0.8w)+1.02$ which is positive for 
$w\in [4,6.71]$. Hence $H^*>0$ in $[0.3,1/2]\times [4,6.71]$. 
\end{proof}

\begin{lm}\label{lma5caseE3}
The inequality $a_5\geq 0$ fails for 
$(t,w)\in [2,\infty )\times (0,1/4]\supset \Omega _1$.
\end{lm}

\begin{proof}
The quantity $a_5^*:=256a_5$ equals

$$\begin{array}{l}
1536t+768w-1536t^2-384w^2-1536wt+768w^2t\\ 
+1280wt^2-32w^3t-416w^2t^2-384wt^3-16t^3w^2+16t^4w\\ 
-72t^2w^3-22tw^4-128w^3+512t^3+44w^4-64t^4-96t^5+w^5~.\end{array}$$
The values $v_j$ for $t=2$ of its partial derivatives 
$\partial ^ja_5^*/\partial t^j$, $j=0$, $\ldots$, $5$, equal respectively 

$$\begin{array}{lc}
v_0=-3072-640w^2-480w^3+w^5&,\\ 
v_1=-8192-512w-1088w^2-320w^3-22w^4&,\\ 
v_2=-15360-1280w-1024w^2-144w^3&,\\ v_3=-23040-1536w-96w^2&,\\ 
v_4=-24576+384w&,\\ v_5=-11520&.\end{array}$$
They are all negative for $w\in (0,1/4]$. Hence all coefficients 
of the Taylor series w.r.t. $t$ of the coefficient 
$a_5$ for $t=2$, $w\in (0,1/4]$,  
are negative and such is $a_5$ for 
$(t,w)\in [2,\infty )$ as well. 
\end{proof}

\begin{lm}\label{lm45}
The inequality $a_6\leq 0$ fails for 
$(t,w)\in (0,1/2]\times [6.71,\infty )\supset \Omega _2$ 
and for $(t,w)\in (0,0.3]\times [5,\infty )\supset \Omega _3^+$.
\end{lm}

Thus after Lemmas~\ref{lmH}, \ref{lma5caseE3} and \ref{lm45} 
there remains to prove that for $(t,w)\in \Omega _3^-$ 
the sign(s) of some (of the) coefficient(s) $a_j$ 
is/are not the one(s) prescribed by the sign pattern.

\begin{proof}[Proof of Lemma~\ref{lm45}]
One has 

$$\begin{array}{ccl}
256a_6&=&1024-768w-1536t-576w^2t+1920t^2+864w^2-352w^3\\ 
&&-1280t^3+800t^4-256t^5+26w^4+4w^5-16t^6+384wt\\ 
&&-384wt^2+400w^3t+720w^2t^2+448wt^3-352t^3w^2-256t^4w+40t^3w^3\\ 
&&+104t^4w^2+64t^5w-272t^2w^3-t^2w^4-56tw^4-2tw^5~.\end{array}$$
We list below the values of the functions 
$u_j:=256\, \partial ^ja_6/\partial w^j$, 
$j=0$, $\ldots$, $5$, for $w=6.71$. They are all positive for $t\in (0,1/2]$ 
(this can be checked numerically). From the Taylor series of 
$a_6$ for $w=6.71$ one concludes that $a_6>0$ for 
$(t,w)\in (0,1/2]\times [6.71,\infty )$. Here's the list: 

$$\begin{array}{ccl}
u_0&:=&-16t^6+173.44t^5+3764.7464t^4-2037.93476t^3\\ 
&&-52440.84297t^2-44774.66948t+35543.86077\\ 
u_1&:=&64t^5+1139.68t^4+1127.0520t^3-28669.71244t^2\\ 
&&-41261.71907t+35244.43996\\ 
u_2&:=&208t^4+906.40t^3-10051.0092t^2-27388.66364t+25772.93608\\ 
u_3&:=&240t^3-1793.04t^2-12021.1320t+12880.8240\\ 
u_4&:=&-24t^2-2954.40t+3844.80\\ u_5&:=&240(2-t)\end{array}$$
In the same way we consider the values for $w=5$ of these same functions, 
see the list below. 
One can check that they are all positive for $t\in (0,0.3]$ and 
by analogy we conclude that $a_6>0$ for 
$(t,w)\in (0,0.3]\times [5,\infty )$.

$$\begin{array}{ccl}
u_0&:=&-16t^6+64t^5+2120t^4-2840t^3-16625t^2-5266t+3534\\  
u_1&:=&64t^5+784t^4-72t^3-14084t^2-9626t+6972\\ 
u_2&:=&208t^4+496t^3-7020t^2-10952t+8968\\ 
u_3&:=&240t^3-1752t^2-7320t+7008\\ 
u_4&:=&-24t^2-2544t+3024\\ u_5&:=&240(2-t)\end{array}$$
\end{proof}

\begin{lm}\label{lmdecreasing}
For $(t,w)\in (0,1/2]\times [4,6.71]\supset \Omega _3^-$ 
the coefficient $a_6$ is a decreasing 
function in $t$. For $t=0$, $w\in [4,6.71]$ one has $a_6\geq 0$ with 
equality only for $w=4$. 
\end{lm}

\begin{proof}
The second claim of the lemma follows from 

$$256\,a_6|_{t=0}=4w^5+26w^4-352w^3+864w^2-768w+1024$$
whose real roots are $-13.978\ldots$, $3.110\ldots$ and $4$. To prove 
the first claim we list the derivatives 
$\eta _j:=256\, \partial ^ja_6/\partial t^j|_{t=0}$, $j=1$, $\ldots$, $6$, 
and their real roots 
($\eta _4$ has no real roots):

$$\begin{array}{ccl}
\eta _1&:=&-2w^5-56w^4+400w^3-576w^2+384w-1536\\ &&-34.115\ldots ~~~,~~~
2.782\ldots ~~~,~~~4\\ 
\eta _2&:=&-2w^4-544w^3+1440w^2-768w+3840\\ 
&&-274.626\ldots ~~~,~~~2.948\ldots \\ 
\eta _3&:=&240w^3-2112w^2+2688w-7680\\ &&7.894\ldots \\ 
\eta _4&:=&2496w^2-6144w+19200\\ \eta _5&:=&7680w-30720\\ &&4\\ 
\eta _6&:=&-11520~.\end{array}$$
As we see, for $w\in [4,6.71]$ one has $\eta _1\leq 0$, 
$\eta _2<0$, $\eta _3<0$, $\eta _4>0$, $\eta _5\geq 0$ and $\eta _6<0$. 
One can majorize the Taylor series for $t=0$ of 

$$256\, \partial a_6/\partial t=\eta _1+t(\eta _2+t\eta _3/2+
t^2\eta _4/6+t^3\eta _5/24+t^4\eta _6/120)$$ 
by omitting the nonpositive terms $\eta _1$, $t^2\eta _3/2$ and $t^5\eta _6/120$ 
and by giving to $t$ inside the brackets its maximal value $1/2$. 
This gives the polynomial 

$$t(\eta _2+\eta _4/24+\eta _5/192)=t(-2w^4-544w^3+1544w^2-984w+4480)~,$$
with real roots $-274.815\ldots$ and $3.083\ldots$, hence negative 
for $w\in [4,6.71]$. 
\end{proof}

\begin{lm}\label{lmtau}
Consider the quantity $H^*$ (see (\ref{H})) 
as a polynomial in $t$. For $w\in [4,6.71]$ 
it has a single root $\tau (w)\in [0,1/2]$, 
$$\tau =(-4w^2+5w-4+\sqrt{(4w^2+19w+4)(4w^2-13w+4)})/4(4w-1)~.$$ 
One has $H^*<0$ (hence $a_{10}<0$) for $t<\tau$ and 
$H^*>0$, $a_{10}>0$ for $t>\tau$. The equality $\tau =0$ takes place only 
for $w=4$. 
\end{lm}  

\begin{proof}
The statements about $\tau$ are to be checked directly. The signs of $H^*$ 
follow easily from $H^*|_{t=0}=2w(4-w)\leq 0$ with equality only for $w=4$.
\end{proof}

\begin{lm}\label{lmtaubis}
Consider $a_6$ as a function in $(t,w)$. 
Then with $\tau$ as defined in Lemma~\ref{lmtau} one has $a_6(\tau ,w)\geq 0$ 
for $w\in [4,5]$ with equality only for $w=4$.
\end{lm}

\begin{rem}
{\rm The lemma implies that at least one of the inequalities $a_6<0$ 
and $a_{10}<0$ fails in $\Omega _3^-$. Indeed, for $t\geq \tau$ this is 
$a_{10}<0$ (see Lemma~\ref{lmtau}), for $t<\tau$ this is $a_6<0$ 
(see Lemmas~\ref{lmdecreasing} and \ref{lmtaubis}).}
\end{rem} 

\begin{proof}[Proof of Lemma~\ref{lmtaubis}]
Set $Y:=\sqrt{(4w^2+19w+4)(4w^2-13w+4)}$. 
One checks numerically that 

$$\begin{array}{ccl}
256\, a_6(\tau ,w)&=&(wC_0+(4w^2+19w+4)C_1Y)/(4w-1)^6~~~{\rm ,~~where}\\ \\ 
C_0&:=&6144w^{10}-6144w^9-224512w^8+2284416w^7\\ &&-6369192w^6+6270368w^5 
-3922014w^4
+1993629w^3\\ &&-860272w^2+234384w-25728\\ 
C_1&:=&384w^7-2496w^6+632w^5-4064w^4\\ &&+4730w^3-1355w^2-136w+64~.\end{array}$$
(With $t=\tau (w)$, $a_6$ becomes a degree $6$ polynomial in $Y$ with 
coefficients in $\mathbb{R}(t)$. Using the fact that $Y^2$ is a polynomial 
in $t$ one obtains the above form of $256\, a_6$.) 
All real roots of $C_0$ are smaller than $4$, so $C_0>0$ for $w\in [4,5]$. 
The real roots of $C_1$ equal $-0.192\ldots$, $0.269\ldots$ and $6.455\ldots$, 
so $C_1$ is negative for $w\in [4,5]$. Hence $wC_0-(4w^2+19w+4)C_1Y>0$ 
and the inequality $wC_0+(4w^2+19w+4)C_1Y>0$ 
is equivalent to $w^2C_0^2-(4w^2+19w+4)^2C_1^2Y^2>0$. 
The left-hand side of the last 
inequality equals $128(w-4)C_2(4w-1)^6$ with 

$$\begin{array}{ccl}
C_2&:=&55296w^{12}+82944w^{11}-1638912w^{10}+6310368w^9\\ 
&&-13847224w^8+10530920w^7-8336710w^6+5520431w^5\\ 
&&-2256796w^4+758480w^3-378304w^2+63488w+2048~.\end{array}$$
The largest real root of $C_2$ equals $3.045\ldots <4$, so $C_2>0$ for 
$w\in [4,5]$ and the lemma is proved.
\end{proof}
\end{proof}

\section{Proofs of Lemmas~\protect\ref{lmH4ter}, 
\protect\ref{lmH5} and \protect\ref{lmH6}
\protect\label{prlmter}}

\begin{proof}[Proof of Lemma~\ref{lmH4ter}:]
\begin{nota}\label{notaP}
{\rm If $\zeta _1$, $\zeta _2$, $\ldots$, $\zeta _k$ are distinct roots 
of the polynomial $P$ (not necessarily simple), then by $P_{\zeta _1}$, 
$P_{\zeta _1,\zeta _2}$, $\ldots$, $P_{\zeta _1,\zeta _2,\ldots ,\zeta _k}$ we denote the 
polynomials $P/(x-\zeta _1)$, $P/(x-\zeta _1)(x-\zeta _2)$, $\ldots$, 
$P/(x-\zeta _1)(x-\zeta _2)\ldots (x-\zeta _k)$.}
\end{nota}

Denote by $u$, $v$, $w$ and $t$ the four distinct roots of $P$ (all nonzero). 
Hence $P=(x-u)^m(x-v)^n(x-w)^p(x-t)^q$, $m+n+p+q=11$. 
For $j=1$, $5$ or $6$ we show that the Jacobian $3\times 4$-matrix 
$J:=(\partial (a_{10},a_9,a_j)/\partial (u,v,w,t))^t$ 
(where $a_{10}$, $a_9$, $a_j$ are the corresponding coefficients of $P$ 
expressed as functions of $(u,v,w,t)$) is of rank $3$. 
(The entry in position $(2,3)$ of $J$ is $\partial a_9/\partial w$.) Hence 
one can vary the values of $(u,v,w,t)$ in such a way that $a_{10}$ and $a_j$ 
remain fixed (the value of $a_{10}$ being $-1$) and $a_9$ takes all possible 
nearby values. Hence the polynomial is not $a_9$-maximal. 

The columns of $J$ are defined by the coefficients of the polynomials 
$-mP_u=\partial P/\partial u$, $-nP_v$, $-pP_w$ and $-qP_t$. 
By abuse of language 
we say that the linear space $\mathcal{F}$ spanned 
by the columns of $J$ is generated by the polynomials 
$P_u$, $P_v$, $P_w$ and $P_t$. As $P_{u,v}=(P_u-P_v)/(v-u)$, 
$P_{u,w}=(P_u-P_w)/(w-u)$ and $P_{u,t}=(P_u-P_t)/(t-u)$, 
one can choose as generators of $\mathcal{F}$ 
the quadruple ($P_u$, $P_{u,v}$, $P_{u,w}$, $P_{u,t}$); in the same 
way ($P_u$, $P_{u,v}$, $P_{u,v,w}$, $P_{u,v,t}$) or 
($P_u$, $P_{u,v}$, $P_{u,v,w}$, $P_{u,v,w,t}$) (the latter polynomials 
are of respective 
degrees $10$, $9$, $8$ and $7$). As $(x-t)P_{u,v,w,t}=P_{u,v,w}$, 
$(x-w)P_{u,v}=P_{u,v,w}$ etc. one can choose as generators the quadruple 
($x^3P_{u,v,w,t}$, $x^2P_{u,v,w,t}$, $xP_{u,v,w,t}$, $P_{u,v,w,t}$). 
Set $P_{u,v,w,t}:=x^7+Ax^6+\cdots +G$. The coefficients 
of $x^{10}$, $x^9$ and $x^6$ of the last quadruple define the matrix 
$J^*:=\left( \begin{array}{cccc}1&0&0&0 \\A&1&0&0 \\ D&C&B&A\end{array}
\right)$. 
Its columns span the space $\mathcal{F}$ hence rank$\, J^*=$rank$\, J$. 
As at least one of 
the coefficients $B$ and $A$ is nonzero (Lemma~\ref{lmhyp}) one has 
rank$\, J^*=3$ and the lemma follows (for the case $j=6$). In the 
cases $j=5$ and 
$j=1$ the last row of $J^*$ equals respectively $(~E~D~C~B~)$ and 
$(~0~0~G~F~)$ and in the same way rank$\, J^*=3$.

\end{proof}

\begin{proof}[Proof of Lemma~\ref{lmH5}:]
We are using Notation~\ref{notaP} and the method of proof of 
Lemma~\ref{lmH4ter}. Denote by $u$, $v$, $w$, $t$, $h$ the 
five distinct real roots of $P$ (not necessarily simple). 
Thus using Lemma~\ref{lmHtriple} one can assume that

\begin{equation}\label{possible}
P=(x+u)^{\ell}(x+v)^m(x+w)^n(x-t)^2(x-h)~~~,~~~u, v, w, t, h > 0~~~,~~~
\ell +m+n=8~.\end{equation}
Set 
$J:=(\partial (a_{10},a_9,a_j,a_1)/\partial (u,v,w,t,h))^t$, $j=5$ or $6$. 
The columns of $J$ span a linear space $\mathcal{L}$ defined by 
analogy with the space $\mathcal{F}$ of the proof of Lemma~\ref{lmH4ter}, 
but spanned by $4$-vector-columns.

Set $P_{u,v,w,t,h}:=x^6+ax^5+bx^4+cx^3+dx^2+fx+g$. Consider the vector-column 
$$(0,0,0,0,1,a,b,c,d,f,g)^t~.$$
The similar vector-columns defined after the polynomials $x^sP_{u,v,w,t,h}$, 
$s\leq 4$, are obtained from this one by successive shifts 
by one position upward. To obtain generators of $\mathcal{L}$ one has 
to restrict these vector-columns to the rows corresponding to $x^{10}$ 
(first), $x^9$ (second), $x^j$ ($(11-j)$th) and $x$ (tenth row). 

Further we assume that $a_1=0$. 
If this is not the case, then at most one of the 
conditions $a_5=0$ and $a_6=0$ is fulfilled and the proof of the lemma can be 
finished by analogy with the proof of Lemma~\ref{lmH4ter}. 

Consider first the case $j=6$. 
Hence the rank of $J$ is the same as the 
rank of the matrix 

$$M:=\left( \begin{array}{ccccc}1&0&0&0&0\\ a&1&0&0&0\\ d&c&b&a&1\\ 
0&0&0&g&f\end{array}\right) ~~~
\begin{array}{c}x^{10}\\ x^9\\ x^6\\ x\end{array}~.$$
One has rank$\, M=2+$rank$\, N$, where 
$N=\left( \begin{array}{ccc}b&a&1\\ 0&g&f\end{array}\right)$. Given that 
$g\neq 0$, one can have rank$\, N<2$ only if $b=0$ and $af=g$. We show that the 
condition $b=0$ leads to the contradiction $a_{10}>0$. 
We set $u=1$ to reduce the number 
of parameters, so we require to hold only the inequality $a_{10}<0$, 
but not the equality $a_{10}=-1$. We have to consider the following cases for 
the values of the triple $(\ell ,m,n)$ 
(see (\ref{possible})): 

\noindent 1) $(6,1,1)$, 2) $(5,2,1)$, 3) $(4,3,1)$, 4) $(4,2,2)$ 
and 5) $(3,3,2)$. Notice that 

$$P_{1,v,w,t,h}=(x+1)^{\ell -1}(x+v)^{m-1}(x+w)^{n-1}(x-t)~.$$
In case 1) one has $b=10-5t$, so $t=2$. For $t=2$ one has 
$a_1=4vw-20vwh-4hv-4hw$ and the condition $a_1=0$ yields 
$h=h_1:=vw/(5vw+v+w)<1/5$. 
Notice that $a_{10}=2+v+w-h$ which for $h=h_1$ is positive -- a contradiction. 

In case 2) we obtain $b=6u^2+4uv-4ut-tv$ hence $t=t_2:=2(3+2v)/(4+v)$. 
One has $a_1=-tv(-vwt-2vwh+thv+5thvw+2thw)$ and for $t=t_2$ the 
condition $a_1=0$ gives 

$$h=h_2:=vw(3+2v)/(9v^2w+3v+2v^2+15vw+6w)<w~.$$
Observe that $a_{10}=5+2v-2t+(w-h)>5+2v-2t$. However for $t=t_2$ one has 
$5+2v-2t_2=(8+5v+2v^2)/(4+v)>0$.

In case 3) one gets $b=3+6v+v^2-3t-2tv=0$, so $t=t_3:=(3+6v+v^2)/(3+2v)$. 
As $a_1=-tv^2(-vwt-2vwh+thv+4thwv+3thw)=0$, for $t=t_3$ one obtains 

$$h=h_3:=vw(3+6v+v^2)/(24vw+23v^2w+3v+6v^2+v^3+4wv^3+9w)<w~.$$
One has $a_{10}=4+3v-2t+(w-h)>4+3v-2t$. For $t=t_3$ one checks directly 
that 

$$4+3v-2t_3=(6+5v+4v^2)/(3+2v)>0~~~,~~~{\rm i.e.}~~a_1>0~.$$ 

In case 4) one has $b=3+3v+3w+vw-3t-tv-tw$, therefore 
$t=t_4:=(3+3v+3w+vw)/(3+v+w)$. As $a_1=-tvw(-vwt-2vwh+4thwv+2thv+2thw)$, 
for $t=t_4$ it follows from $a_1=0$ that 

$$h=h_4:=\frac{vw(3+3v+3w+vw)}{2(9vw+6v^2w+6vw^2+2v^2w^2+3v+3v^2+3w+3w^2)}$$
which is $<w/2$. One has $a_{10}=4+2v+2w-2t-h$ which for $h=h_4$ and $t=t_4$ is 

$$>4+2v+3w/2-2t_4=(1/2)(12+8v+5w+4v^2+3vw+3w^2)/(3+v+w)>0~.$$
In case 5) one has $b=1+4v+v^2+2w+2vw-2t-2tv-tw$, therefore 

$$t=t_5:=(1+4v+v^2+2w+2vw)/(2+2v+w)~.$$ 
As $a_1=-tv^2w(-vwt-2vwh+3thwv+2thv+3thw)$, the condition $a_1=0$ yields 

$$h=h_5:=\frac{vw(1+4v+v^2+2w+2vw)}
{15vw+15v^2w+10vw^2+3wv^3+6v^2w^2+2v+8v^2+2v^3+3w+6w^2}$$
which is $<w/2$. One has $a_{10}=3+3v+2w-2t-h$ which for $t=t_5$, $h=h_5$ is 

$$>3+3v+3w/2-2t_5=(1/2)(8+8v+4w+8v^2+4vw+3w^2)/(2+2v+w)>0~.$$

Now consider the case $j=5$. The matrices $M$ and $N$ equal respectively 

$$M:=\left( \begin{array}{ccccc}1&0&0&0&0\\ a&1&0&0&0\\ f&d&c&b&a\\ 
0&0&0&g&f\end{array}\right) ~~~~,~~~
N=\left( \begin{array}{ccc}c&b&a\\ 0&g&f\end{array}\right) ~.$$
One has rank$\, N<2$ only for $c=0$ and $bf=ag$. Similarly to the case $j=6$ 
we show that the equality $c=0$ leads to the contradiction $a_{10}>0$. 
We define the cases 
1) -- 5) in the same way as above.

In case 1) one has $c=10-10t$, so $t=1$. As $a_1=vw-4vwh-hv-hw$, the equality 
$a_1=0$ implies $h=h^1:=vw/(4vw+v+w)<1/4$. One has $a_{10}=4+v+w-h$ which for 
$h=h^1$ is positive -- a contradiction.

In case 2) one gets $c=-2u(-2u^2-3uv+3ut+2tv)$, so $c=0$ implies 
$t=t^2:=(2+3v)/(3+2v)$. From $a_1=-kv(-vwk-2vwh+thv+5thwv+2thw)=0$ one gets 
(for $t=t^2$) 

$$h=h^2:=vw(2+3v)/(11v^2w+2v+3v^2+10vw+4w)<w~.$$
From $a_{10}=5+2v+w-2t-h$ one sees that for $h=h^2$, $t=t^2$ it is true that 

$$a_{10}>5+2v-2t^2=(11+10v+4v^2)/(3+2v)>0~.$$

In case 3) one obtains $c=1+6v+3v^2-3t-6tv-v^2t$, so 
$t=t^3:=(1+6v+3v^2)/(3+6v+v^2)$. The condition 
$a_1=-tv^2(-vwt-2vwh+thv+4thwv+3thw)=0$ with $t=t^3$ implies 

$$h=h^3:=vw(1+6v+3v^2)/(16vw+21v^2w+10wv^3+v+6v^2+3v^3+3w)<w~.$$
But then from $a_{10}=4+3v+w-2t-h$ with $t=t^3$, $h=h^3$ follows 

$$a_{10}>4+3v-2t^3=(10+21v+16v^2+3v^3)/(3+6v+v^2)>0~.$$

In case 4) one has $c=1+3v+3w+3vw-3t-3tv-3tw-vwt$, so $c=0$ implies 
$t=t^4:=(1+3v+3w+3vw)/(3+3v+3w+vw)$. For $t=t^4$ the condition 
$a_1=-tvw(-vwt-2vwh+4thwv+2thv+2thw)=0$ implies 

$$h=h^4:=(1/2)vw(1+3v+3w+3vw)/(5vw+6v^2w+6vw^2+5v^2w^2+v+3v^2+w+3w^2)$$
which is $<w/2$. Thus $a_{10}=4+2v+2w-2t-h$ with $t=t^4$, $h=h^4$ implies 

$$\begin{array}{ccl}a_{10}&>&4+2v+3w/2-2t^4\\ 
&=&\displaystyle{\frac{20+24v+21w+17vw+12v^2+4v^2w+9w^2+3vw^2}
{2(3+3v+3w+vw)}>0~.}\end{array}$$

In case 5) we get $c=2v+2v^2+w+4vw+v^2w-t-4tv-v^2t-2tw-2vwt$ and $c=0$ implies 

$$t=t^5:=(2v+2v^2+w+4vw+v^2w)/(1+4v+v^2+2w+2vw)~.$$
For $t=t^5$ the equalities $a_1=-tv^2w(-vwt-2vwh+3thwv+2thv+3thw)=0$ yield 

$$h=h^5:=\frac{vw(2v+2v^2+w+4vw+v^2w)}
{6vw+12v^2w+6wv^3+11vw^2+11v^2w^2+3w^2v^3+4v^2+4v^3+3w^2}$$
which is $<w/2$. 
Hence $a_{10}=3+3v+2w-2t-h$ with $t=t^5$, $h=h^5$ implies 

$$\begin{array}{ccl}a_{10}&>&3+3v+3w/2-2t^5\\ 
&=&\displaystyle{\frac{6+22v+22v^2+11w+20vw+6v^3+11v^2w+6w^2+6vw^2}
{2(1+4v+v^2+2w+2vw)}>0~.}\end{array}$$

\end{proof}

\begin{proof}[Proof of Lemma~\ref{lmH6}:]
We use the same ideas and notation as in the proof of Lemma~\ref{lmH5}. 
Six of the six or more real roots of $P$ are denoted by $(u,v,w,t,h,q)$. 
The space $\mathcal{L}$ is defined by analogy with the one of the 
proof of Lemma~\ref{lmH5}. 
The Jacobian matrix $J$ is of the form 

$$J:=(\partial (a_{10},a_9,a_j,a_1)/\partial (u,v,w,t,h,q))^t~.$$ 
Set $P_{u,v,w,t,h,q}:=x^5+ax^4+bx^3+cx^2+dx+f$ and consider  
the vector-column  

$$(0,0,0,0,0,1,a,b,c,d,f)^t~.$$
Its successive shifts by one position upward correspond to the polynomials 
$x^sP_{u,v,w,t,h,q}$, $s\leq 5$. In the case $j=6$ 
the matrices $M$ and $N$ 
look like this:

$$M=\left( \begin{array}{cccccc}1&0&0&0&0&0\\ a&1&0&0&0&0\\ 
d&c&b&a&1&0\\ 0&0&0&0&f&d\end{array}\right)~~~\, {\rm and}~~~\, 
N=\left( \begin{array}{cccc}b&a&1&0\\ 0&0&f&d\end{array}\right)~.$$
One has rank$\, M=2+$rank$\, N$ and rank$\, N=2$ 
because $f\neq 0$ and at least one of 
the two coefficients $b$ and $a$ is nonzero (Lemma~\ref{lmhyp}). Hence 
rank$\, M=4$ and the lemma is proved by analogy with Lemmas~\ref{lmH4ter} and 
\ref{lmH5}. 
In the case $j=5$ the third row of $M$ equals $(~f~d~c~b~a~1~)$, the first 
row of $N$ equals $(~c~b~a~1~)$, at least one of the two coefficients $c$ 
and $b$ is nonzero and again rank$\, M=4$.

\end{proof}

\end{document}